\documentclass[a4j,10pt,onecolumn,oneside,notitlepage,openany,final]{report}
\usepackage{amsmath}
\usepackage{amsthm}
\usepackage{amsfonts}
\usepackage{amssymb}
\usepackage{bm}
\usepackage{graphicx}
\usepackage{longtable}

\newtheorem{thm}[]{Theorem} 

\newtheorem{Def}[thm]{Definition}
\newtheorem{lem}[thm]{Lemma}
\newtheorem{cor}[thm]{Corollary}
\newtheorem{prp}[thm]{Proposition}
\newtheorem{rem}[thm]{Remark}
\newtheorem{example}[thm]{Example}

\begin{document}
\title{\textbf{\mathversion{bold} 
Topology of the punctual Hilbert schemes of plane curve singularities with a single Puiseux pair}}
\author{Masahiro WATARI}
\date{}
\maketitle

\begin{center}
\textit{Dedicated to Professor Fumio Sakai on his 70th birthday}
\end{center}

\begin{abstract}
Piontkowski proved the existence of affine cell decompositions of Jacobian factors of plane curve singularities with a single Puiseux pair.
He also provided a combinatorial description of the Euler numbers and Betti numbers of these Jacobian factors.
Following his results, Oblomkov, Rasmussen, and Shende demonstrated the existence of affine cell decompositions of punctual Hilbert schemes for the same type of singularity.
In the present paper, we revisit their theorem from a computational perspective and describe the Euler numbers and Betti numbers of the punctual Hilbert schemes. 
\end{abstract}

\noindent
\textbf{Keywords}: plane curve singularity, punctual Hilbert scheme, Euler number, Betti number\\
\textbf{Mathematics Subject Classification (2020)} 14C05, 14G10, 14H20
\section{Introduction} \label{Introduction}
Let $X$ be a singular plane curve over $\mathbb{C}$ with its singular point $o$. 
We refer to the pair $(X,o)$ as a plane curve singularity.  
In this paper, we consider the case where $(X,o)$ is reduced and irreducible. 
Let $R:=\widehat{\mathcal{O}}_{X,o}$ denote the completion of  the local ring $\mathcal{O}_{X,o}$ of $X$ at $o$, and let $\Gamma$ be the semigroup of $R$.
We denote by $\mathrm{Hilb}^r (X,o)$ the punctual Hilbert scheme of degree $r$ for the given singularity $(X,o)$. 
It is known that $\mathrm{Hilb}^r (X,o)$ admits a decomposition of the form 
\begin{equation}\label{decomposition 1}
\mathrm{Hilb}^r (X,o)=\bigsqcup_{\Delta\in \mathrm{Mod}_r(\Gamma)}H(\Delta) 
\end{equation}
where $\mathrm{Mod}_r(\Gamma):=\{\Delta|\,\text{$\Delta$ is a $\Gamma$-semimodule with $\#(\Gamma\setminus \Delta)=r$}\}$. 
In general, some of $H(\Delta)$ are empty. 
We  call the component $H(\Delta)$  the \emph{$\Delta$-subset} of $\mathrm{Hilb}^r (X,o)$. 

Piontkowski \cite{P} studied the topology of the Jacobian factors of irreducible plane curve singularities with a single Puiseux pair. 
It is known that a Jacobian factor $J_{X,o}$ of an irreducible plane curve singularity $(X,o)$ is isomorphic to certain punctual Hilbert schemes (see Corollary $\ref{stable}$ below). 
Therefore, $J_{X,o}$ also admits a decomposition (\ref{decomposition 1}). 
By showing that each $\Delta$-subset $H(\Delta)$ in the decomposition is isomorphic to an affine space,   
he described the Euler number and the Betti numbers of $J_{X,o}$.  
Following his result, Oblomkov, Rasmussen, and Shende proved the following theorem:

\begin{thm}[\cite{ORS}, Theorem 13]\label{Watari's result 1} 
Let $p$ and $q$ be coprime positive integers with $p<q$. 
For an irreducible plane curve singularity $(X,o)$ with $\Gamma=\langle p, q\rangle$, each $\Delta$-subset $H(\Delta)$ 
of $\mathrm{Hilb}^r (X,o)$
is isomorphic to an affine space of dimension  
\begin{equation}\label{dimension of H 2}
\sum_{i=1}^{p-1}\#[\{(-\min\Delta +\Gamma )\cap[a_i,a_i+q)\}\setminus \Delta^{(0)}]
\end{equation}
where $\{a_0,\ldots,a_{p-1}\}$ is the $p$-basis of the $0$-normalization $\Delta^{(0)}:=-\min\Delta +\Delta$ of $\Delta$. 
\end{thm}

\begin{rem}
Although the dimension {\rm (\ref{dimension of H 2})} was originally expressed differently in {\rm \cite{ORS}},
we re-express it here using the $p$-basis, as in {\rm \cite{P}}.
\end{rem}

In this paper, we apply Piontokowski's method, emphasizing the computational aspects, to prove Theorem \ref{Watari's result 1}.
The following result immediately follows:

\begin{cor}\label{Watari's result 2} 
Under the same assumption as in Theorem {\rm \ref{Watari's result 1}},  
the Euler number of $\mathrm{Hilb}^r (X,o)$ is equal to $\# \mathrm{Mod}_r(\Gamma)$. 
\end{cor}

We also give a combinatorial description of the Euler numbers and Betti numbers of punctual Hilbert schemes, which generalize Piontkowski's result. 
For $\Delta\in\mathrm{Mod}_r(\Gamma)$, define the codimension of $H(\Delta)$ as 
$
\mathrm{codim}\, H(\Delta):=\dim \mathrm{Hilb}^r (X,o) -\dim H(\Delta).
$ 
We also set
\begin{align*}
\mathcal{H}_{r,d}:&=\{H(\Delta)|\,\Delta\in \mathrm{Mod}_r (\Gamma) \text{ and }\dim H(\Delta)=d\},\\
\mathcal{H}^d_r:&=\{H(\Delta)|\,\Delta\in \mathrm{Mod}_r (\Gamma) \text{ and }\mathrm{codim}\, H(\Delta)=d\}.
\end{align*}

\begin{thm}\label{Watari's result 3} 
Let $(X,o)$ be the same as in Theorem {\rm \ref{Watari's result 1}}. 
The odd $($co-$)$ homology groups of  $\mathrm{Hilb}^r (X,o)$ vanish. 
The even $($co-$)$ homology groups of  $\mathrm{Hilb}^r (X,o)$ are free abelian groups with Betti numbers
$$
\text{$h_{2d}(\mathrm{Hilb}^r (X,o))=\# \mathcal{H}_{r,d}$ and $h^{2d}(\mathrm{Hilb}^r (X,o))=\# \mathcal{H}^d_r$.}
$$
\end{thm}

\begin{rem}\label{Remark 1}
In { \rm \cite{P}}, Piontkowski originally proved  Theorem $\ref{Watari's result 1}$, Corollary $\ref{Watari's result 2}$, and Theorem $\ref{Watari's result 3}$ for the case where $r\ge c$.
\end{rem}

This paper is structured as follows: 
Sections \ref{preliminaries}, \ref{Semi-modules}, and \ref{Grobner bases and Syzygies} summarize the foundational concepts required for the proofs. 
Section \ref{preliminaries} recalls the properties of punctual Hilbert schemes of curve singularities. 
Section \ref{Semi-modules} introduces definitions and facts related to $\Gamma$-semimodules. 
Section \ref{Grobner bases and Syzygies} discusses computational techniques involving Gröbner bases. 
In Section \ref{Proof}, we prove the main theorems. 
Section \ref{Examples} presents examples, and Section \ref{Remarks} provides remarks on related results.

\section{Punctual Hilbert schemes of curve singularities}\label{preliminaries}
In this section, we recall properties of punctual Hilbert schemes established in \cite{PS}. 
Although we focus on irreducible plane curve singularities here, the notions presented here hold in more general situations. 

Let $(X,o)$  be an irreducible plane curve singularity. 
By Puiseux's theorem, there exist  coprime positive integers $p$ and $q$ such that 
$
R=\mathbb{C}[[t^p, \phi]]
$
where $\phi=t^q+\text{higher order terms}$ and $p<q$. 
The normalization $\overline{R}$ of $R$ is isomorphic to $\mathbb{C}[[t]]$. 
Let $\nu$ be the natural valuation $\nu:\overline{R}\setminus \{0\} \rightarrow \mathbb{Z}_{\ge 0}$  defined by $\nu (f)=\mathrm{ord}_t(f)$, and set $\nu(0)=\infty$. 
We call $\Gamma:=\nu (R)$ the \emph{semigroup} of $R$. 
A  semigroup $\Gamma$ generated by positive integers $\gamma_1,\gamma_2,\ldots,\gamma_l$ is denoted as 
 $\Gamma=\langle\gamma_1,\gamma_2,\ldots,\gamma_l\rangle$. 
The \emph{$\delta$-invariant} of $R$ is defined as $\delta:=\dim _{\mathbb{C}}(\overline{R}/R)$. 
The \emph{conductor} $c$ of $\Gamma$ is the smallest integer in $\Gamma$ such that $c-1\notin \Gamma$ and $c+n\in \Gamma$ for any $n\in \mathbb{N}$. 
It satisfies $\delta+1\le c\le 2\delta$ and $c=2\delta$ if and only if $R$ is Gorenstein (see \cite{Serre}). 

For a positive integer $a$, we denote by $(t^a)$ an ideal of $\overline{R}$.
Let $\mathrm{Gr}\left(\delta,\overline{R}/(t^{2\delta})\right)$ be the Grassmannian which consists of $\delta$-dimensional linear subspaces of $\overline{R}/(t^{2\delta})$. 
For $M\in\mathrm{Gr}\left(\delta,\overline{R}/(t^{2\delta})\right)$, we define a multiplication by $R\times M\ni (f,m+(t^{2\delta}))\mapsto fm+(t^{2\delta})\in M$. 
The set 
$$
J_{X,o}:=\bigg\{M\in \mathrm{Gr}\left(\delta,\overline{R}/(t^{2\delta})\right)\bigg|\,
\begin{array}{l}
\text{$M$ is an $R$-submodule with respect to}\\
\text{ the above multiplication.}
\end{array}
\bigg\}
$$
is called the \emph{Jacobian factor} of $(X,o)$, 
which was introduced by Rego in \cite{Re}. 
For any non-negative integer $r$, set
$$
\mathcal{I}_r:=\{I\subset R|\,\text{$I$ is an ideal of $R$ with $\dim R/I=r$}\}.
$$

Observe that any element $I$ in $\mathcal{I}_r$ satisfies $(t^{r+2\delta}) \cap R \subseteq I\subseteq (t^r)\cap R$. 
It follows that $(t^{2\delta})\subseteq t^{-r}I \subseteq \overline{R}$ and $\dim_{k}(\overline{R}/t^{-r}I)=\dim_{k}(\overline{R}/I)-r=\delta$. 
Using these properties, Pfister and Steenbrink defined a map 
\begin{equation}\label{phai}
\varphi_r: \mathcal{I}_r\rightarrow J_{X,o}\subset \mathrm{Gr}(\delta, \overline{R}/(t^{2\delta}))
\end{equation}
by $\varphi _r (I)=t^{-r}I/(t^{2\delta})$ (see \cite{PS}). 
We call $\varphi _r$ the \emph{$\delta$-normalized embedding}. 
It has the following properties:
\begin{prp}[\cite{PS}, Theorem\,3]\label{thm3}
The $\delta$-normalized embedding $\varphi_r$  is injective for any non-negative integer $r$. 
It is also a bijection for $r\ge c$. 
The punctual Hilbert scheme $\mathrm{Hilb}^r (X,o)$  is Zariski closed in 
$J_{X,o}$. 
\end{prp}

\begin{Def}
We call $\mathrm{Hilb}^r (X,o):=\varphi_r(\mathcal{I}_r)$ the punctual Hilbert scheme of degree $r$ for $(X,o)$. 
\end{Def}

The following fact follows from Proposition\,\ref{thm3}:

\begin{cor}\label{stable}
For any positive integer $r$ with $r\ge c$, we have  
$$\mathrm{Hilb}^r(X,o)\cong \mathrm{Hilb}^c (X,o)\cong J_{X,o}.$$
\end{cor}

\section{$\Gamma$-semimodules}\label{Semi-modules}
Let $\Gamma$ be a semigroup.   
A subset $\Delta\subset \mathbb{Z}$ is called a \emph{$\Gamma$-semimodule} if it satisfies $\Delta+\Gamma\subseteq \Delta$.   
It is straitforward to see that, for any ideal $I$ of $R$, the set $\Gamma(I):=\{\nu(f)|\,f\in I\}$ is a $\Gamma$-semimodule.
If integers $a_1,\ldots,a_s$ generate  $\Delta$ (i.e., $\Delta=\sum_{i=1}^s(a_i+\Gamma)$ holds), 
we write $\Gamma=\langle a_1,\ldots,a_s\rangle_\Gamma$.   
Two $\Gamma$-semimodules $\Delta_1$ and $\Delta_2$ are \emph{isomorphic} if there exists an integer $a$ such that $\Delta_1=a+\Delta_2=\{a+d|\,d\in \Delta_2\}$. 
We define two special normalizations of a $\Gamma$-semimodule $\Delta$: 
We define the \emph{$0$-normalization} of $\Delta$ as $\Delta^{(0)}:=-\min \Delta+\Delta$.  
In general, $\Delta$ is called \emph{$0$-normalized} if $\min {\Delta}=0$. 
On the other hand, $\Delta$ is said to be \emph{$\delta$-normalized} if $\# (\mathbb{N}\setminus \Delta)=\delta$. 
For $I\in\mathcal{I}_r$, there exists an element $\Delta$ of $\mathrm{Mod}_r(\Gamma)$ such that $\Delta=\Gamma(I)$. 
From the inclusion $(t^{2\delta})\subset t^{-r}I$ and the definition of $\varphi_r$ in (\ref{phai}), it follows  that 
$-r+\Delta$ is $\delta$-normalized, and $\#\{(-r+\Delta)\cap [0,2\delta-1]\}=\delta$ holds. 
We call $-r+\Delta$ the  \emph{$\delta$-normalization} of $\Delta$ and denote it by $\Delta^{(\delta)}$. 

In \cite{P}, Piontkowski introduced special generators of a $\Gamma$-semimodule. 
\begin{Def}[\cite{P}]
Let $\Gamma=\langle p,q\rangle$ where $p<q$ and $\gcd(p,q)=1$. 
The \emph{$p$-basis} of a $\Gamma$-semimodule $\Delta$ is the unique set $\{a_0,a_1,\ldots, a_{p-1}\}$ satisfying
\begin{equation*}
\Delta=\bigcup_{i=0}^{p-1}(a_i+p\mathbb{N}) \text{\  \ and\  \ } a_i\equiv iq \text{ $(\mathrm{mod}\ p)$}.
\end{equation*}
In particular, we have $\Delta=\langle a_0,\ldots,a_{p-1}\rangle_\Gamma$.
\end{Def}

Throughout this paper, we consider the $p$-basis only for $0$-normalized $\Gamma$-semimodules and 
therefore we always set $a_0=0$. 
Using the inclusion $\mathbb{N}q\subset \Gamma\subset \Delta^{(0)}$, we find non-negative integers $\alpha_1,\ldots,\alpha_{p-1}\in\mathbb{N}$ such that 
\begin{equation}\label{alphas}
a_0=0,\,a_1=q-\alpha_1p,\,a_2=2q-\alpha_2p,\,\ldots ,\,a_{p-1}=(p-1)q-\alpha_{p-1}p.
\end{equation}
We also set $\alpha_0=0$. 
These satisfy $0\le \alpha_1 \le \alpha_2\le\cdots \le \alpha_{p-1}<q. $

The set $\mathcal{I}_r$ can be decomposed in terms of $\Gamma$-semimodules.
\begin{lem}[\cite{SW1}, Proposition 6]\label{decomposition 2}
We have
\begin{equation}\label{decomposition of ideal}
\mathcal{I}_r=\bigsqcup_{\Delta\in \mathrm{Mod}_r(\Gamma)}\mathcal{I}(\Delta) 
\end{equation}
where $\mathcal{I}(\Delta):=\{I\in \mathcal{I}_r|\,\Gamma(I)=\Delta\}$. 
\end{lem}
By setting $H(\Delta):=\varphi_r(\mathcal{I}(\Delta))$,  the stratification (\ref{decomposition 1}) of $\mathrm{Hilb}^r (X,o)$ follows from Lemma \ref{decomposition 2}.

\begin{rem}
In general, some of the components in $(\ref{decomposition of ideal})$ may be empty. 
However, in {\rm \cite{P}}, Piontkowski showed that all components are not empty for an irreducible curve singularity with $\Gamma=\langle p, q\rangle$.
\end{rem}

\section{Gr\"{o}bner bases}\label{Grobner bases and Syzygies}
We begin by recalling some key facts about Gröbner bases, primarily following \cite{HH}. 
Let $R=\mathbb{C}[[x_1(t),\ldots,x_l(t)]]$ be a subring of $\mathbb{C}[[t]]$ such that $\dim_{\mathbb{C}}\mathbb{C}[[t]]/R<\infty$, and let $M\subset \mathbb{C}[[t]]$ be an $R$-module. 
As in Section \ref{preliminaries}, let $\nu$ be the natural valuation $\nu:\mathbb{C}[[t]]\setminus \{0\} \rightarrow \mathbb{Z}_{\ge 0}$. 
We also set $\nu(0)=\infty$. 
Put $\Gamma:=\nu(R)$ and $\Gamma(M):=\nu(M)$. 
Clearly, $\Gamma(M)$ is a $\Gamma$-semimodule. 
We consider  the local order 
$\nu(1) \succ\nu  (t)\succ\nu  (t^2)\succ \cdots$. 
For $f\in M$, we denote by ${\small \mathrm{LC}}(f)$ (resp. ${\small \mathrm{LT}}(f))$ the leading coefficient of $f$ (resp. the leading term of $f$) with respect to this order. 

\begin{Def}\label{Def of SAGBI}
A subset $G=\{g_1,\ldots,g_m\}$ of $R$ is called a \emph{SAGBI basis} $($Subalgebra Analog to Gr\"{o}bner Bases for Ideals$)$ 
if, for any $f\in R$, there exists a multi-index $(\beta_1, \ldots,\beta_m)\in \mathbb{Z}^m_{\ge 0}$
such that $\mathrm{\small LT}(f)=\mathrm{\small LT}(g_1^{\beta_1}\cdots g_m^{\beta_m})$.
\end{Def}

Definition \ref{Def of SAGBI} implies the following:  

\begin{thm}
A SAGBI basis $G$ of R generates $R$ as $\mathbb{C}$-algebra. 
\end{thm}


\begin{Def}\label{Def of SB}
Let $G$ be a SAGBI basis for $R$ and let $H$ be a subset of $M$. 
The pair $(G,H)$ is called a \emph{standard basis} of $M$ if,  for any $f\in M$, there exist $h\in H$ 
and a multi-index $(\beta_1, \ldots,\beta_m)\in \mathbb{Z}^m_{\ge 0}$
such that $\mathrm{\small LT}(f)=\mathrm{\small LT}(g_1^{\beta_1}\cdots g_m^{\beta_m}h)$.
\end{Def}

Similar to the above, the following theorem holds:

\begin{thm}
Let  $(G,H)$  be a standard basis of $M$. 
The set $H$ generates $M$ as $R$-module.
\end{thm}

The following proposition follows from Definition\,\ref{Def of SB}. 
\begin{prp}\label{referee's prp2}
Let $G=\{g_1,\ldots,g_m\}\subset R$ be a SAGBI basis for $R$ and let $H=\{h_1,\ldots,h_n\}$ be a subset of $M$. 
Then $(G,H)$ is  a standard basis of $M$ if and only if 
$\Gamma=\langle \nu(g_1),\ldots, \nu(g_m)\rangle$ 
and  $\Gamma(M)=\langle \nu(h_1),\ldots, \nu(h_n)\rangle _{\Gamma}$. 
\end{prp}

\begin{thm}
Let $G$ and $H$ be as in Proposition {\rm \ref{referee's prp2}}. 
If $(G,H)$ is a standard basis of $M$, then any $f\in \mathbb{C}[[t]]$ can be written as 
\begin{equation}\label{Division}
f=q_1h_1+\cdots+q_nh_n+r
\end{equation}
where $q_1,\ldots,q_n\in\mathbb{C}[[g_1,\ldots,g_m]]$ and
$
r=\sum_{j\notin \Gamma(M)}c_jt^{j}. 
$
\end{thm}

We call $r$ in (\ref{Division}) the \emph{reduction} of $f$ modulo $(G,H)$, and denote it by $R(f,G,H)$.

\vspace{-1mm}
\noindent
\begin{proof}
We claim that the following algorithm yields (\ref{Division}):\\
\vspace{-3mm}

\noindent
\textbf{Division Algorithm}\\
\texttt{Input:} $f\in R$, $G=\{g_1,\ldots,g_m\}$, $H=\{h_1,\ldots,h_n\}$\\
\texttt{Output:} $q_1,\ldots,q_n$, $r$
\vspace{1mm}

\noindent
\texttt{Define:} $q_1:=0,\ldots,q_n:=0$, $r=0$, $p:=f$\\
\texttt{WHILE} $p\neq0$ \texttt{DO}\\
\quad If $\nu(p)\in\Gamma(M)$, \texttt{THEN} find smallest $l$ such that \\
\quad $\mathrm{\small LT}(p)=\mathrm{\small LT}(g_1^{\alpha_1}\cdots g_m^{\beta_m}h_l)\text{ for some  }(\beta_1, \ldots,\beta_m)\in \mathbb{Z}^m_{\ge 0}$\\
\qquad $q_l:=q_l+g_1^{\beta_1}\cdots g_m^{\beta_m}$\\
\qquad $p:=p-g_1^{\beta_1}\cdots g_m^{\beta_m}h_l$\\
\qquad $r:=r$\\
\quad \texttt{ELSE}\\
\qquad $p:=p-\mathrm{\small LT}(p)$\\
\qquad $r:=r+\mathrm{\small LT}(f)$\\
\vspace{-3mm}

\noindent
We allow the algorithm to proceed infinitely many steps. 
Since the order $\nu(p)$ strictly increases at every step, $p$ converges to 0. 
It is also obvious that the orders of all terms in the final reduction $r$ are not in $\Gamma(M)$. 
For further details, refer to the proof of Theorem 3 in Chapter 2, Section 3 of \cite{CLO}. 
\end{proof}

The reduction $R(f,G,H)$ has the following property:

\begin{lem}
Let $(G, H)$ be a standard basis of $M$. 
For any $f\in M$, the reduction $R(f,G,H)$ is unique 
no matter how the elements of $G$ and $H$ are listed. 
\end{lem}
\begin{proof}
The proof of this lemma is similar to that of Proposition 1 in Chapter 2, Section 6 of \cite{CLO}. 
So we omit it.     
\end{proof}
For $f_1,f_2\in M$, there exist multi-indices $(\beta_1,\ldots,\beta_m), (\gamma_1,\ldots,\gamma_m)\in \mathbb{Z}^m_{\ge 0}$ and elements $h_1,h_2\in H$ such that $\mathrm{\small LT}(f_1)=\mathrm{\small LT}(g_1^{\beta_1}\cdots g_m^{\beta_m}h_1)$ and 
$\mathrm{\small LT}(f_2)=\mathrm{\small LT}(g_1^{\gamma_1}\cdots g_m^{\gamma_m}h_2)$. 
The $S$-\emph{process} for $f_1,f_2\in M$ is defined as 
\begin{equation*}
S(f_1,f_2):=g_1^{\gamma_1}\cdots g_m^{\gamma_m}h_2 f_1- g_1^{\beta_1}\cdots g_m^{\beta_m}h_1 f_2.
\end{equation*}

\noindent
Among all such expressions, we define the one with minimal order as the \emph{minimal}  $S$-\emph{process} of $f_1$ and $f_2$, denoted by $S_{\mathrm{min}}(f_1,f_2)$.  

\begin{prp}[\cite{HH}, Theorem\,2.3]\label{referee's prp}
Let $G=\{g_1,\ldots,g_m\}\subset R$ be a SAGBI basis for $R$ and let $H=\{h_1,\ldots,h_n\}$ be a subset of $M$. 
The pair $(G,H)$ is  a standard basis of $M$ if and only if $R(S_{\mathrm{min}}(h_i,h_j), G,H)=0$ for all $h_i,h_j\in H$. 
\end{prp}

\section{Proofs of the Main Results}\label{Proof}
As noted in Remark \ref{Remark 1}, Piontkowski \cite{P} originally proved Theorem $\ref{Watari's result 1}$, Corollary $\ref{Watari's result 2}$, and Theorem $\ref{Watari's result 3}$ in the case $r\ge c$. 
The following result about syzygies was used in his proof:
\begin{lem}[\cite{P}, Proposition 5]\label{Syzygies}
Let $\Gamma=\langle p, q\rangle$, and let $\Delta=\cup^{p-1}_{i=0}(a_i+p\mathbb{N})$ be a $\Gamma$-semimoudlue with $p$-basis $\{a_0,\ldots,a_{p-1}\}$. 
Consider a graded algebra $\mathbb{C}[\Gamma]:=\bigoplus_{\gamma\in \Gamma}\mathbb{C}t^\gamma$,  
and  a graded $\mathbb{C}[\Gamma]$-algebra $\mathbb{C}[\Delta]:=\bigoplus_{\gamma\in \Delta}\mathbb{C}t^\gamma$
 generated by $(t^{a_0},\ldots,t^{a_{p-1}})$.
Then the syzygies of this $p$-tuple  are minimally generated by the following vectors$:$
\begin{align*}
v_0:&=(t^q,-t^{\alpha_1p},0,\ldots,0),\\
v_1:&=(0,t^q,-t^{(\alpha_2-\alpha_1)p},0,\ldots,0),\\
&\ \, \vdots\\
v_{p-2}:&=(0,\ldots,0,t^q,-t^{(\alpha_{p-1}-\alpha_{p-2})p}),\\
v_{p-1}:&=(-t^{(q-\alpha_{p-1})p},0,\ldots,0,t^q)
\end{align*} 
where $\alpha_1,\ldots,\alpha_{p-1}$ are numbers given in $(\ref{alphas})$.
\end{lem}
Before proving the main theorems, we briefly summarize Piontkowski's approach. 
\vspace{1mm}

\noindent
\textbf{Outline of Piontkowski's proofs.}
If $r\ge c$, then we have $\mathrm{Hilb}^r(X,o)\cong \mathrm{Hilb}^c (X,o)\cong J_{X,o}$ by Corollary \ref{stable}. 
Recall that points of $\mathrm{Hilb}^c (X,o)$ correspond bijectively to elements $I\in\mathcal{I}_c$ via $\varphi_c$.  
Moreover, each $I\in\mathcal{I}_c$ corresponds to $t^{-d}I$ where $d:=\min I$.
If $\Delta:=\nu(I)$, then $\nu(t^{-d}I)=\Delta^{(0)}$. 
Thus, Piontkowski identified $t^{-d}I$ with  a point in the $\Delta$-subset $H(\Delta)$ of $J_{X,o}=\mathrm{Hilb}^c (X,o)$. 
In his proof, the following generators of $t^{-d}I$ were used:
\begin{align}\label{g1}
    h_i&=t^{a_i}+\sum_{\tiny a_{i}+k\notin\Delta^{(0)},\ k\ge 0}\lambda_{i,k}t^{a_i+k},\ (i=0,\ldots,p-1)
\end{align}
where $\{a_0,a_1,\ldots,a_{p-1}\}$ is the $p$-basis of $\Delta^{(0)}$. 
He verified the conditions for
 $\Delta^{(0)}= \langle \nu(h_0), \ldots, \nu(h_{p-1})\rangle_\Gamma$.
Set $G=\{t^{p},\phi\}$ where $\phi=t^q+\text{hight order terms}$, and $H=\{h_0, \ldots, h_{p-1}\}$. 
Then
$\Delta^{(0)}= \langle \nu(h_0), \ldots, \nu(h_{p-1})\rangle_\Gamma$ holds if and only if $(G,H)$ is  a standard basis of $R$-module $t^{-d}I$ by Proposition \ref{referee's prp2}. 
Moreover, by Proposition \ref{referee's prp} and Lemma \ref{Syzygies}, 
$(G,H)$ is  a standard basis of $t^{-d}I$ if and only if 
\begin{align}\label{Equation 1}
&R(S_{\mathrm{min}}(h_i,h_{i+1}), G,H)=0 \text{ for } i=0,\ldots,p-2,\\\label{Equation 2}
&R(S_{\mathrm{min}}(h_{p-1},h_{0}), G,H)=0
\end{align} 
hold. 
The minimal $S$-processes are given by
\begin{align*}
S_{\mathrm{min}}(h_i,h_{i+1})&=\phi h_i-t^{(\alpha_{i+1}-\alpha_i)p}h_{i+1}\text{ for } i=0,\ldots,p-2,\\
S_{\mathrm{min}}(h_{p-1},h_{0})&=\phi h_{p-1}-t^{(q-\alpha_{p-1})p}h_0.
\end{align*} 
Using (\ref{Equation 1}) and (\ref{Equation 2}), he analyzed the coefficients of $S$-processes and concluded that the number of independent coefficients in $h_i$'s such that $\Delta^{(0)}= \langle \nu(h_0), \ldots, \nu(h_{p-1})\rangle_{\Gamma}$ is given by 
\begin{equation}\label{dimension of H 3}
\sum_{i=0}^{p-1}\#\{[a_i,a_i+q]\setminus \Delta^{(0)}\}.
\end{equation}
Hence, each $\Delta$-subset $H(\Delta)$ of $J_{X,o}=\mathrm{Hilb}^c (X,o)$ is an affine cell of dimension (\ref{dimension of H 3}). 
When $r\ge c$, we see that $\min \Delta\ge c$, and so  $(-\min\Delta +\Gamma )\cap[a_i,a_i+q]=[a_i,a_i+q]$, making 
(\ref{dimension of H 2}) equal to (\ref{dimension of H 3}). 
 
Once the decomposition (\ref{decomposition 1}) is shown to be an affine cell decomposition, 
it follows that Euler number of $J_{X,o}$ equals the number of $\Delta$-subsets in (\ref{decomposition 1}). 
Moreover, these $\Delta$-subsets form a CW complex. 
So the Betti numbers follow from standard (co-)homology theory (see \cite{P} for details). $\square$

\vspace{1mm}
\noindent
\textbf{Proofs of Theorem $\ref{Watari's result 1}$, Corollary $\ref{Watari's result 2}$ and Theorem $\ref{Watari's result 3}$.}
Piontkowski's method applies directly to our general case without modification.  
We note one remark about the difference between the dimensions  (\ref{dimension of H 2}) and (\ref{dimension of H 3}) of $\Delta$-subsets. 
Let $I\in I(\Delta)\subset\mathcal{I}_r$, and set $d=\min\{I\}$. 
Consider the generators $t^{d}h_i$ $(i=0,\ldots,p-1)$ of $I$, which can be written as
\begin{equation}\label{generators of I}
t^{d}h_i=t^{b_i}+\sum_{\tiny b_{i}+k\notin\Delta,\ k\ge 0}\lambda_{i,k}t^{b_i+k}
\end{equation}
where $b_i:=a_i+d$ and $(i=0,\ldots,p-1)$.
Every ideal in $I(\Delta)$  has generators of the same form as (\ref{generators of I}). 
The numbers of coefficients in (\ref{g1}) and (\ref{generators of I}) are same, since the sets $\mathbb{N}\setminus \Delta^{(0)}$ and $\{n\in \mathbb{N}\setminus \Delta|\,n>d\}$ are in one-to-one correspondence. 
Piontokowski's idea was to count the number of independent coefficients in (\ref{g1}) 
(equivalently, in (\ref{generators of I})). 
This count gives the dimension of $H(\Delta)$ as an affine space. 
If $r<c$, then (\ref{generators of I}) may contain some terms whose orders are not in $\Gamma$. 
We claim that the coefficients of such terms are not independent. 
Setting $A:=\{\beta_i\notin \Gamma|\,\beta_i>q\}:=\{\beta_1,\ldots,\beta_s\}$, 
we may assume that $\phi$ is of form $t^q+\sum_{\beta_i\in A} c_it^{\beta_i}$.  
The set $\{t^p, \phi\}$ is a standard basis of $R$ in ordinary sense. 
Reducing $t^{d}h_i$ by $\{t^p, \phi\}$, 
the final reduction of $t^{d}h_i$ must be 0, since $t^{d}h_i\in R$.
This implies that, for each term  $\lambda_{i,k}t^{b_{i}+k}$ with $b_i+k\notin \Gamma$,  
there exists polynomial $f_{i,k}(x_1,\ldots,x_s)$ in $\mathbb{C}[x_1,\ldots,x_s]$ such that 
$\lambda_{i,k}=f_{i,k}(c_1,\ldots,c_s)$  
(i.e. the coefficient $\lambda_{i,k}$  is determined by $\varphi$, and hence are not independent). 
This fact explains the difference between  (\ref{dimension of H 2}) and (\ref{dimension of H 3}). 
 $\square$
\section{Examples}\label{Examples}
The punctual Hilbert schemes of curve singularity of type $A_{2l}$  $($i.e.,  the curve singularity with $R=\mathbb{C}[[t^{2},t^{2l+1}]])$ was studied in \cite{SW2}. 
For this singularity, we have $\Gamma=\langle 2, 2l+1\rangle$, $\delta=l$, and $c=2l$. 
Let $e(\mathrm{Hilb}^r (X,o))$ denote  the Euler number of $\mathrm{Hilb}^r (X,o)$. 
Let [ $\cdot$ ]  denote the greatest integer function, where for a real number $a$, the value [$a$] is the largest integer satisfying $[a] \le a$.
The following example is obtained from the results in \cite{SW2}:
\begin{example}
Let $(X,o)$ be the $A_{2l}$-singularity. 
The Euler numbers of the punctual Hilbert schemes $\mathrm{Hilb}^r (X,o)$ are given in the following table$:$
\begin{center}
\begin{tabular}{c|cc}
$r$&$0\le r\le 2l-1$&$r\ge 2l$\\
\hline
$e(\mathrm{Hilb}^r (X,o))$ &$[r/2]+1$&$l+1$ 
\end{tabular}
\end{center}
The Betti numbers of  $\mathrm{Hilb}^r (X,o)$ are as follows$:$
\begin{equation*}
h_{2d}(\mathrm{Hilb}^r(X,o))= h^{2d}(\mathrm{Hilb}^r(X,o))=1\quad (d\in\{0,1,\ldots,[r/2]\})
\end{equation*}
\end{example}

The punctual Hilbert schemes of curve singularities of type $E_{6}$ and $E_{8}$  $($i.e. the curve singularities with $R=\mathbb{C}[[t^{3},t^{4}]]$ and $R=\mathbb{C}[[t^{3},t^{5}]]$ respectively) were studied in \cite{SW1}. 
For the $E_{6}$-singularity, the fundamental invariants are  $\Gamma=\langle 3,4\rangle$, $\delta=3$ and $c=6$. 
On the other hand, we have $\Gamma=\langle 3,5\rangle$, $\delta=4$ and $c=8$. 
The following examples are derived from the results in \cite{SW1}:

\begin{example}
Let $(X,o)$ be the $E_{6}$-singularity. 
The Euler numbers of the punctual Hilbert schemes $\mathrm{Hilb}^r (X,o)$ are given in the following table$:$

\begin{center}
\begin{tabular}{c|ccccccc}
$r$&$0$&$1$&$2$&$3$&$4$&$5$&$r\ge 6$\\
\hline
$e(\mathrm{Hilb}^r (X,o))$&$1$&$1$&$2$&$3$&$4$&$4$&$5$
\end{tabular}
\end{center}

The Betti numbers of  $\mathrm{Hilb}^r (X,o)$ are also given in the following tables$:$

\begin{table}[h]
\begin{minipage}[t]{.45\textwidth}
\begin{flushright}
\begin{tabular}{c|cccc}
$r$&$h_0$&$h_2$&$h_4$&$h_6$\\
\hline
$0$&$1$&&&\\
\hline
$1$&$1$&&&\\
\hline
$2$&$1$&$1$&&\\
 \hline
$3$&$1$&$1$&1&\\
\hline
$4$&$1$&$1$&2&\\
 \hline
$5$&$1$&$1$&2&\\
\hline
$6$&$1$&$1$&2&1
\end{tabular}
\end{flushright}
\end{minipage}
\hfill
\begin{minipage}[t]{.45\textwidth}
\begin{flushleft}
\begin{tabular}{c|cccc}
$r$&$h^0$&$h^2$&$h^4$&$h^6$\\
\hline
$0$&$1$&&&\\
\hline
$1$&$1$&&&\\
\hline
$2$&$1$&$1$&&\\
 \hline
$3$&$1$&$1$&1&\\
\hline
$4$&$2$&$1$&1&\\
 \hline
$5$&$2$&$1$&1&\\
\hline
$6$&$1$&$2$&1&1
\end{tabular}
\end{flushleft}
\end{minipage}
\end{table}
\end{example}

\begin{example}
Let $(X,o)$ be the $E_{8}$-singularity. 
The Euler numbers of the punctual Hilbert schemes $\mathrm{Hilb}^r (X,o)$ are given in the following table$:$

\begin{center}
\begin{tabular}{c|ccccccccc}
$r$&$0$&$1$&$2$&$3$&$4$&$5$&$6$&$7$&$r\ge 8$\\
\hline
$e(\mathrm{Hilb}^r (X,o))$&$1$&$1$&$2$&$3$&$4$&$5$&$6$&$6$&$7$
\end{tabular}
\end{center}

The Betti numbers of  $\mathrm{Hilb}^r (X,o)$ are also given in the following tables$:$

\begin{table}[h]
\begin{minipage}[t]{.45\textwidth}
\begin{flushright}
\begin{tabular}{c|ccccc}
$r$&$h_0$&$h_2$&$h_4$&$h_6$&$h_8$\\
\hline
$0$&$1$&&&&\\
\hline
$1$&$1$&&&&\\
\hline
$2$&$1$&$1$&&&\\
 \hline
$3$&$1$&$1$&1&&\\
\hline
$4$&$1$&$1$&2&&\\
 \hline
$5$&$1$&$1$&2&1&\\
\hline
$6$&$1$&$2$&1&2&\\
\hline
$7$&$1$&$1$&$2$&$2$&\\
\hline
$8$&$1$&$1$&$2$&$2$&$1$
\end{tabular}
\end{flushright}
\end{minipage}
\hfill
\begin{minipage}[t]{.45\textwidth}
\begin{flushleft}
\begin{tabular}{c|ccccc}
$r$&$h^0$&$h^2$&$h^4$&$h^6$&$h^8$\\
\hline
$0$&$1$&&&&\\
\hline
$1$&$1$&&&&\\
\hline
$2$&$1$&$1$&&&\\
 \hline
$3$&$1$&$1$&1&&\\
\hline
$4$&$2$&$1$&1&&\\
 \hline
$5$&$1$&$2$&1&1&\\
\hline
$6$&$2$&$1$&2&1&\\
 \hline
$7$&$2$&$2$&1&1&\\
\hline
$8$&$1$&$2$&2&1&1
\end{tabular}
\end{flushleft}
\end{minipage}
\end{table}
\end{example}

\section{Remarks on the results similar to Theorem \ref{Watari's result 1}}\label{Remarks}
In this section, we discuss the results of Pfister and Steenbrink in \cite{PS}. 
Here we consider a monomial curve singularity $(X,o)$. 
\begin{Def}[\cite{PS}, Definition 8]\label{monomial curve singularity}
We call a curve singularity $(X,o)$ with $R=\mathbb{C}[[t^{a_1},\cdots,t^{a_n}]]$ for some positive integers $a_1,\ldots,a_n$ a \emph{monomial curve singularity}.
\end{Def}
Without loss of generality, we assume that $\gcd(a_1,\ldots,a_n)=1$ in Definition \ref{monomial curve singularity}.  
Additionally, we focus on special semigroups defined as follows:
\begin{Def}[\cite{PS}, Definition 9]
A semigroup $\Gamma$ is called \emph{monomial} if $0\in\Gamma$, $\# (\mathbb{N}\setminus \Gamma)<\infty$, and any reduced, irreducible curve singularity with $\Gamma$ is a monomial curve singularity. 
\end{Def}

The monomial semigroups were completely determined in \cite{PS}.

\begin{thm}[\cite{PS}, Theorem 10]
A monomial semigroup is monomial if and only if it is one of the following three types$:$
\begin{align*}
&\{im|\,i=0,1,\ldots,s\}\cup [sm+b,\infty )\text{ with } 1\le b<m, \ s\ge 1,\\
&\{0\}\cup [m,m+r-1]\cup [m+r+1,\infty)\text{ with } 2\le r\le m-1,\\ 
&\{0,m\}\cup [m+2,2m]\cup [2m+2,\infty)\text{ with } m\ge 3
\end{align*}
\end{thm}

\begin{example}
Semigroups $\langle 2,2l+1\rangle$ with $l\ge 1$, $\langle 3,4\rangle$, and $\langle 3,5\rangle$ are all monomial semigroups that can be realized as plane curve singularities. 
\end{example}

For a $\Gamma$-semimodule $\Delta$, let $S$ be the set of minimal generators of $\Delta^{(\delta)}$. 
Set $S':=S\cap [0, 2\delta-1]$. 
For each $\gamma\in S'$, define $J_\gamma:=[\gamma+1,2\delta-1]\setminus \Delta^{(\delta)}$.

\begin{thm}[\cite{PS}, Theorem 11]\label{PS Theorem 11}
Let $(X,o)$ be a curve singularity with a monomial semigroup. 
For  $I\in\mathcal{I}_r$, there exist uniquely determined $u_{\gamma,j}\in \mathbb{C}$  
such that, as $R$-submodule, $\varphi_r(I)$ is generated by 
\begin{equation*}
g_{\gamma}=t^{\gamma}+\sum_{j\in J_{\gamma}}u_{\gamma,j}t^j, (\gamma\in S') .
\end{equation*}
\end{thm}

\begin{cor}[\cite{PS}, Corollary of Theorem 11]\label{PS Corollary of Theorem 11}
For a curve singularity $(X,o)$ with a monomial semigroup, a $\Delta$-subset $H(\Delta)$ of $\mathrm{Hilb}^r(X,o)$ is an affine space of dimension $\sum_{\gamma\in S'}\# J_\gamma$. 
\end{cor}

\begin{rem}
Theorem \ref{Watari's result 1} is stated in terms of $\Delta^{(0)}$, whereas  Theorem \ref{PS Theorem 11} and Corollary \ref{PS Corollary of Theorem 11} are expressed using $\Delta^{(\delta)}$.
\end{rem}

However, Theorem \ref{PS Theorem 11} and Corollary \ref{PS Corollary of Theorem 11} do not hold in general. 
Suppose that Theorem \ref{PS Theorem 11} and Corollary \ref{PS Corollary of Theorem 11} do hold for the $E_6$-singularity. 
Recall that $R=\mathbb{C}[[t^3,t^4]]$, $\Gamma=\langle 3,4\rangle$ and $\delta=3$ for this singularity. 
Consider a $\Gamma$-semimodule $\Delta= \langle 4,6,7\rangle_\Gamma$. 
It is easy to verify that $\Delta\in\mathrm{Mod}_2(\Gamma)$. 
So its $\delta$-normarization is give by  $\Delta^{(\delta)}=-2+\Delta=\langle 2,4,5\rangle_\Gamma$. 
Let $I$ be an element of $\mathcal{I}(\Delta)$.
It follows from Theorem \ref{PS Theorem 11} that $\varphi_2(I)$ is generated by $g_{2}=t^2+u_{2,3}t^3$, $g_4=t^4$ and $g_5=t^5$ as $R$-submodule.  
By Corollary \ref{PS Corollary of Theorem 11}, we also have $\dim H(\Delta)=1$. 
However, we see that $\mathcal{I}(\Delta)=\{(t^4, t^6,t^7)\}$, 
which means $H(\Delta)$ consist of a single point. 
Therefore, $\dim H(\Delta)=0$. 
Moreover, it is obvious that $\varphi_2((t^4, t^6,t^7))$ is generated by $t^2$, $t^4$, and $t^5$.
These contradict the assumption.
 
\section*{Acknowledgements}
The author would like to thank Vivek Shende for informing him about his result.

\noindent
\small 
Masahiro Watari\\
 University of Kuala Lumpur, Malaysia France Institute\\
Japanese Collaboration Program\\
Section 14, Jalan Damai, Seksyen 14, 43650\\
Bandar Baru Bangi, Selengor, Malaysia.\\
E-mail:masahiro@unikl.edu.my
\end{document}